\documentclass[12pt,letterpaper]{amsart}
\usepackage{cite}
\usepackage{graphicx}
\usepackage[T1]{fontenc}
\usepackage{amssymb,amscd,amsthm,amsxtra}
\usepackage{latexsym}
\usepackage{epsfig}
\usepackage{esint}
\usepackage{enumerate}
\usepackage{enumitem}
\usepackage{esint}
\usepackage{enumitem}
\usepackage{amsmath}
\usepackage{bm}
\usepackage{amssymb}
\usepackage{amsthm}
\usepackage[usenames,dvipsnames,svgnames,table]{xcolor}
\usepackage[linktocpage=true,colorlinks=true,linkcolor=Blue,citecolor=BrickRed,urlcolor=RoyalBlue]{hyperref}
\usepackage{comment}
\usepackage[colorinlistoftodos,prependcaption,textsize=tiny]{todonotes}
\usepackage{caption}
\usepackage{amscd}
\usepackage[mathscr]{euscript} 
\usepackage{mathtools, verbatim}
\usepackage[utf8]{inputenc}
\usepackage{indentfirst}
\usepackage{enumerate}
\usepackage[colorlinks=true]{hyperref}
\hypersetup{urlcolor=blue, citecolor=red}
\usepackage{graphicx}
\usepackage[usenames,dvipsnames]{pstricks}
\usepackage{epsfig}
\usepackage{pst-grad} 
\usepackage{pst-plot} 
\usepackage{ifthen}
\usepackage{color, xcolor}

\usepackage{mathrsfs} 

\usepackage{pifont}

\newlist{thmlist}{enumerate}{1}
\setlist[thmlist]{label=(\alph{thmlisti}), ref=\thethm.(\alph{thmlisti}),noitemsep}

\usepackage{thmtools}

\declaretheorem[
    name=Theorem,
    Refname={Theorem,Theorems},
    numberwithin=section]{thm}
\declaretheorem[
    name=Lemma,
    Refname={Lemma,Lemmas},
    sibling=thm]{lem}

    \declaretheorem[
    name=Proposition,
    Refname={Proposition,Propositions},
    sibling=thm]{prop}

\usepackage{nameref,hyperref}
\usepackage[capitalize]{cleveref}

\Crefname{thm}{Theorem}{Theorems}
\Crefname{lem}{Lemma}{Lemmas}

\makeatletter
\newcommand{\mylabel}[2]{#2\def\@currentlabel{#2}\label{#1}}
\makeatother

\numberwithin{equation}{section} 

\newtheorem{theorem}{Theorem}

\newtheorem{definition}{Definition}
\newtheorem{lemma}{Lemma}
\newtheorem{corollary}{Corollary}

\theoremstyle{remark}
\newtheorem{remark}{Remark}[section]

\newcommand{\D}{\mathrm{div }}

\def \suchthat {\ \big | \ }

\title[Semilinear Equations with Singular Potentials]{Geometric Estimates for Solutions of Semilinear Equations with Singular Potentials}

\author[T. M. Nascimento and Lei Zhang]{Thialita M. Nascimento and Lei Zhang }
\address{Department of Mathematics, University of Florida, Little Hall, Gainesville, FL 32601}{}
\email{t.nascimento@ufl.edu}

\email{leizhang@ufl.edu}


\thanks{The second author is partially supported by Simons Foundation Collaboration Grant SFI-MPS-TSM-00013752. 
}

\begin{document}

\date{} 

\begin{abstract} 

In this work, we study local minimizers of elliptic functionals with strong absorption terms and unbounded, sign-changing sources. These problems naturally interpolate between two classical free boundary problems: Bernoulli-type (cavity) and obstacle-type. While previous studies  have focused on bounded and strictly positive sources, we extend sharp regularity and nondegeneracy estimates to the unbounded, sign-changing setting, providing a comprehensive analysis of how the underlying nonlinearity interacts with minimal integrability assumptions on the source.

\tableofcontents

\end{abstract}
\maketitle

\section{Introduction}

In this work we study nonnegative solutions of the  semilinear elliptic equation
\begin{equation}\label{main equation}
   \mathrm{div}(A(x)\nabla u) = \gamma f(x) u^{\gamma - 1}
   \quad \text{in } \{ u > 0\} \cap B_1 ,
\end{equation}
where $\gamma \in (0,1)$, $A(x)$ is a uniformly elliptic matrix with measurable coefficients, and 
$f$ belongs to the Lorentz-space $ L^{p,\infty}(B_1)$. 
The nonlinearity $u^{\gamma-1}$ introduces a singular absorption term as $u\to 0$; as a consequence,
solutions may develop a free boundary separating the positivity set $\{u>0\}$ from the zero set $\{u=0\}$.

Equations of this type arise naturally in physical contexts such as combustion 
models, flow through porous media, and chemical reaction processes. From a 
mathematical perspective, problems of the form \eqref{main equation} 
interpolate between two classical free boundary regimes. In the limiting case 
$\gamma = 0$, the equation reduces to a cavity-type free boundary problem, 
which is closely related to the Bernoulli problem studied by Alt and Caffarelli 
\cite{AltCaff81}. Conversely, when $\gamma = 1$, the absorption term behaves 
as the characteristic function $\chi_{\{u > 0\}}$, and the problem becomes 
one of obstacle type. In this sense, equation \eqref{main equation} provides a natural 
bridge between these two fundamental and well-studied classes of free 
boundary problems.

The classical model associated with \eqref{main equation}, corresponding to the case $A \equiv I$ and $f$
positive and bounded, was introduced and studied by Alt and Phillips in \cite{AltPhilips86}.
Under these assumptions, solutions are locally smooth, and the free boundary $\partial\{u>0\}$
enjoys a rich and well-understood geometric structure.
Since then, a large body of work has extended this theory in various directions
(see, e.g., \cite{AraujoTeix2013, Teix2018, MR3980852, AraujoSa22, ATV23}),
but invariably under strong structural assumptions on the source term $f$.
In contrast, many contemporary models involve highly irregular, unbounded,
or even sign-changing sources, for which the classical regularity theory no longer applies
\cite{MR1025884, MR2506429, MR3721865}.

Motivated by these considerations, we study \eqref{main equation} under the  weaker assumption
\[
f \in L^{p,\infty}(B_1),
\]
with no sign restriction on $f$. The Lorentz-space condition identifies the critical integrability  scale under which the singular absorption term remains controlled.

This paper is organized as follows. Section \ref{sct prelim} we introduce the 
mathematical framework and basic tools used throughout the text, and state the main results. In 
Section \ref{sct existence}, we establish the existence and boundedness of 
nonnegative weak solutions. Section \ref{sct local reg} proves the local 
H\"older regularity of minimizers, extending classical results to the 
setting of Lorentz-space sources. Section \ref{sct sharp reg} 
focuses on regularity estimates at the free boundary, yielding sharp growth 
rates. Section \ref{sct nondeg} addresses nondegeneracy by providing 
quantitative lower bounds near the free boundary. Finally, Section 
\ref{sct radial example} illustrates these findings with a radial example 
that demonstrates the sharpness of our estimates.

\section{Assumptions and Main results}\label{sct prelim}

In this section, we introduce the notation, assumptions, and functional framework used throughout the paper.  And we state and comment the main results of this paper. 

\subsection{Notation and assumptions}
Throughout this paper, we denote by $\mathbb{R}^n$ the $n$-dimensional 
Euclidean space and consider an open domain $\Omega \subset 
\mathbb{R}^n$. For a point $x_0 \in \mathbb{R}^n$ and a radius $r > 0$, 
we let $B_r(x_0)$ represent the open ball $\{x \in \mathbb{R}^n : |x - 
x_0| < r\}$. When $x_0 = 0$, we simply write $B_r$. 

We recall that for $p \ge 1$, a measurable function $f$ belongs to the space $L^{p,\infty}(\Omega)$ (also known as the weak $L^p$ or Lorentz space) if
\[
  \| f \|_{L^{p, \infty}(\Omega)} 
  := \sup_{t > 0} \, t \, \big| \{ x \in \Omega : |f(x)| > t \} \big|^{1/p} < \infty,
\]
where $|E|$ denotes the $n$-dimensional Lebesgue measure of a measurable 
set $E \subset \mathbb{R}^n$. These spaces provide a natural framework for handling source terms with  singular growth and arise frequently in the analysis of variational  functionals with singular weights.

 \medskip

\noindent{\bf The variational problem:} We denote by $\mathcal{M}_{\lambda,\Lambda}(B_1)$ the class of symmetric matrix-valued functions
$A : B_1 \to \mathrm{Sym}(n)$ that are uniformly elliptic, i.e.,
\[
\lambda |\xi|^2 \le \langle A(x)\,\xi, \xi\rangle \le \Lambda |\xi|^2 
\quad \text{for all } \xi \in \mathbb{R}^n \text{ and a.e. } x \in B_1,
\]
for some fixed constants $0<\lambda \le \Lambda < \infty$.

In this paper, we consider minimizers of the functional
\begin{equation}\label{main functional}
J_{A,f} (v) := \int_{B_1} \frac{1}{2} \langle A(x) \nabla v, \nabla v \rangle + f(x) (v^+)^\gamma \, dx,
\end{equation}
where $A(x) \in \mathcal{M}_{\lambda,\Lambda}(B_1)$, $\gamma \in (0,1)$, and $f\in L^{p,\infty}(B_1)$, for $p > \frac{n}{2}.$ The minimum is taken over the admissible set
\[
\mathbb{K} :=  \left\{ v \in W^{1,2}(B_1) : v - g \in W_0^{1,2}(B_1) \right\},
\]
where $g \in H^1(B_1)$ is a given boundary datum. For simplicity, we shall write $J$ instead of $J_{A,f}$ whenever there is no ambiguity.  

Throughout the paper, we will assume that the boundary data $g$ satisfies 
\begin{equation}\label{boundary data assumptions}
g \in C(\overline{B_1}), \quad g \ge 0, \; g \not\equiv 0,
\end{equation}
so as to avoid technicalities related to traces and
representatives on $\partial B_1$,  unless otherwise stated.

\medskip
\subsection{Main Results}

The novelty of our work lies in extending the classical theory of Alt and Phillips, \cite{AltPhilips86}, to the setting of Lorentz-type sources, which includes singular potentials of the form $f(x) \sim |x-y_0|^{-n/p}$.  
A fundamental question is whether the regularity of minimizers is preserved in this singular setting. While local H\"older continuity is well-established  for bounded sources $f$, our first main result shows that this property persists for 
$f \in L^{p,\infty}(B_1)$, establishing the necessary analytic foundation 
for investigating the free boundary:

\begin{theorem}[Local H\"older Regularity]\label{Holder reg}
Let $u$ be a local minimizer of \eqref{main functional}. There exist universal constants $0 < \beta \ll 1$ and $C_2 > 0$, such that $u \in C^{0,\beta} (B_{1/2})$ and 
\begin{equation}
    \| u \|_{C^{0, \beta}(B_{1/2})} \le C_2 \| u \|_{L^2 (B_1)}.
\end{equation}
\end{theorem}

A key observation is that local minimizers of $J_{A,f}$ are weak solutions to the  elliptic equation  \eqref{main equation} in the positivity set $\{ u > 0\}$. Consequently, any universal regularity estimate for minimizers must respect the maximum regularity dictated by the singular source term $f \in L^{p,\infty}$ and the nonlinearity exponent $\gamma \in (0,1)$.  
Building on this, we obtain \emph{sharp regularity estimates} at free boundary points, showing that solutions belong to the class $C^{\theta}$ with the optimal exponent
    \[
    \theta = \frac{2 - \frac{n}{p}}{2 - \gamma}.
    \]
This extends the classical Alt--Phillips regularity theory to the setting of unbounded, sign-changing sources (Section \ref{sct sharp reg}). More precisely, we have the following result: 

\begin{theorem}[Regularity at the Free Boundary]\label{improv reg} 
Let $u$ be a nonnegative minimizer of \eqref{main functional} in $B_1$, and let $x_0 \in \partial\{ u > 0 \} \cap B_{1/2}$. Then, 
\[
u(x) \le C \| u \|_{L^2 (B_1)} | x - x_0 |^{\theta} \quad \text{for all } |x - x_0| < \rho_*/4,
\]
where $\rho_* \in (0,1/8)$ and $C>0$ depend only on dimension, ellipticity constants, $\gamma$, and $p$.
\end{theorem}

Compared to the discontinuous regime corresponding to $\gamma=0$, 
recently studied in \cite{SnelTeix24}, 
the singular nonlinearity $(u^+)^{\gamma}$ introduces new challenges  in understanding how solutions detach from the free boundary. In particular, the nonlinearity prevents the direct use of linear barrier arguments.  As a result, the growth of solutions near the free boundary depends  on the local mass of the forcing term.  

 The following theorem addresses these difficulties, providing a quantitative lower bound on the solution near free boundary points where the source carries sufficient mass:

\begin{theorem}[Nondegeneracy]\label{non-deg theom}
Let $u$ be a nonnegative weak solution to \eqref{main equation}. Suppose $f\ge C_0$ for $C_0>0$ in $B_1$, then for any $y_0\in \partial \{u>0\}$ with $B_{2r}(y_0)\in B_1$,
\begin{equation}\label{non-degen-add-1}
\sup_{\partial B_{r}(y_0)} u \ge \tilde{c} r^{\frac{2}{2 - \gamma}},
\end{equation}
where $\tilde{c} > 0$ depends only on $n, \gamma, C_0$, and the ellipticity constants.
\end{theorem}

A key challenge arises at free boundary points $y_0 \in \partial \{u > 0\}$ where the source term $f$ blows up too rapidly. However, if the location  of such a point is known, the assumptions on $f$ in its vicinity can be 
significantly relaxed. Specifically, we demonstrate that the concentration  of mass near $y_0$ compensates for the lack of global boundedness, allowing for a stronger non-degeneracy estimate near $y_0$.

\begin{theorem}\label{non-de-2}
Let $u$ be a nonnegative weak solution to \eqref{main equation}. Suppose  $y_0 \in \partial\{ u > 0\}$ is an interior free boundary point such that $B_{4r_0}(y_0)\in B_1$. If $f$ satisfies the following assumption in $B_{4r_0}(y_0)$: There exist $\tau\in (1-4^{-n},1)$ and $\varepsilon\in (0, 8^{-n})$ such that
\begin{equation}\label{PosMass}
    \left| \{ y \in B_{4r} (y_0 ) : f(y) \ge  r^{-n/p} \}\right| \ge \tau |B_{4r}(y_0)|,  \;  \forall \; 0 < r < r_0
\end{equation}
and 
\[\int_{B_r(y_0)}|f_-(x)|dx\le \varepsilon r^{n-n/p} .\] Then we have
\begin{equation}\label{non-degen-add-2}
\sup_{\partial B_{2r}(y_0)} u \ge \tilde{c} r^{(2 - n/p)/(2 - \gamma)}, \quad 0<r<r_0.
\end{equation}
where $\tilde c>0$ depending on $n$, $\gamma$, $p$ and elliptic constants. 
\end{theorem}
\medskip

\begin{remark}
Theorem \ref{non-deg theom} assumes a positive lower bound on $f$ and provides a non-degeneracy estimate around all interior free boundary points. Theorem \ref{non-de-2} is based on the assumption of $f$ near a known free boundary point and gives the sharp nondegeneracy index. Both theorems are more general than previous results. 
\end{remark}

Our results provide a comprehensive picture of the qualitative behavior of minimizers for the Alt–Phillips equation with singular or sign-changing sources, extending classical free boundary theory to the Lorentz-space setting. In particular, we show that several core features of the Alt–Phillips theory—such as local Hölder regularity, sharp growth near the free boundary, and nondegeneracy—persist far beyond the classical bounded setting, even in the presence of highly singular data and minimal regularity of the coefficients. We expect that these results will serve as a foundation for further developments in the regularity theory of singular free boundary problems, with potential applications in both mathematics and applied sciences.

 \subsection{Auxiliarity results}

\medskip 

In this subsection, we establish foundational results that are utilized repeatedly throughout the  paper.

We begin with a classical embedding theorem  for Lorentz spaces:

\begin{lemma}[Weak embedding]\label{lemma:weak_embedding}
Let $E \subset \mathbb{R}^n$ be a measurable set and $f \in L^{p,\infty}(E)$ for some $p > 0$. Then, for every $0 < r < p$, we have
\begin{equation}\label{embedding weak}
    \| f\|_{L^r (E)} \le \left( \frac{p}{p - r}\right)^{1/r} |E|^{\frac{p-r}{pr}} \, \|f \|_{L^{p,\infty} (E)}.
\end{equation}
\end{lemma}

\begin{proof}
The result follows from the layer-cake representation and the definition of the weak $L^p$ norm:
\begin{align*}
\|f\|_{L^r(E)}^r &= r \int_0^\infty t^{r-1} |\{x \in E : |f(x)| > t\}| \, dt\\
&\le \int_0^{t_1} r  t^{r-1} |E| dt+r\int_{t_1}^{\infty} t^{r-1}
\left( \frac{\|f\|_{L^{p,\infty}(E)}}{t} \right)^p \! dt\\
=&t_1^r |E|+\frac{r}{p-r}\|f\|_{L^{p,\infty}}^p t_1^{r-p}
\end{align*}
By choosing $t_1=\|f\|_{L^{p,\infty}}/|E|^{1/p}$ we obtain \eqref{embedding weak}.
\end{proof}

 \medskip

The next estimate plays a central role in our arguments. 

\begin{lemma}\label{lemma:fu_gamma_estimate}
Assume that $n \ge 2$, $p> \tfrac{n}{2}$, and $\gamma \in (0,2)$. 
Let $E \subset B$ be a measurable set, and let $f \in L^{p,\infty}(E)$. 
Then for any $u \in H^{1}(B)$, the following estimate holds: For $n\ge 3$, there is a constant $C_1(n,p,\gamma)>0$ such that
\begin{equation}\label{fu_gamma_estimate}
\left| \int_E f \, u^\gamma \, dx \right| 
\le C_1 \, 
\| f \|_{L^{p,\infty}(E)} \, \|u \|_{H^1(E)}^\gamma \; |E|^{\frac{2-\gamma}{2} + \frac{\gamma}{n} - \frac{1}{p}} .
\end{equation}
For $n=2$, given $\epsilon>0$, there is a constant $C_2(p,\gamma,\epsilon)>0$ such that 
\begin{equation}\label{f-e-2}
|\int_E f u^{\gamma}dx|\le C_2\|f\|_{L^{p,\infty}(E)}\|u\|_{H^1}^{\gamma}|E|^{1-\frac{1}{p}-\epsilon}.
\end{equation}
\end{lemma}

\begin{proof}
Since $p > \frac{n}{2}$, 
\[\frac 1p+\frac{(n-2)\gamma}{2n}<1\]
for $\gamma\in (0,2)$. 
Let $2^* = \frac{2n}{n-2}$ for $n >2$ and $2^*$ be a large constant for $n =2$. Then by Hölder's inequality and Lemma \ref{lemma:weak_embedding} we have
\begin{eqnarray*}
    \left| \int_E f \, u^\gamma \, dx \right| &\le& \left( \int_E | f(x)|^{\frac{2^*}{2^*-\gamma}}dx \right)^{\frac{2^*-\gamma}{2^*}} \left( \int_{E}|u|^{2^*}\right)^{\frac{\gamma}{2^*}}\nonumber\\
&\le& C \| f \|_{L^{p, \infty} (E)} \, \left( \int_{E}|u|^{2^*}\right)^{\frac{\gamma}{2^*}} |E|^{1 - \frac{\gamma}{2^*} - \frac{1}{p}} \nonumber 
\end{eqnarray*}
Then for $n\ge 3$, there exists $C_1(p,n,\gamma)>0$ such that
\[|\int_E f u^{\gamma}dx | \le C_1\|f\|_{L^{p,\infty}(E)}\|u\|^{\gamma}_{H^1}|E|^{\frac{2-\gamma}{2}+\frac{\gamma}n-\frac 1p}.\]
For $n=2$ choosing $2^*$ as $2^*=\gamma/\epsilon$ we have
\[|\int_E f u^{\gamma}dx |\le C_2(p,\gamma,\epsilon)\|f\|_{L^{p,\infty}(E)}\|u\|_{H^1}^{\gamma}|E|^{1-\frac 1p-\epsilon}.\]
 Here we further note that we choose $\epsilon>0$ small to make $2^*>2$. 
\end{proof}

To conclude this section, we recall some useful inequalities that will be used throughout the paper. For $a,b \ge 0$, the following hold:

$$
    a^r + b^r \le (a + b)^r \le 2^{r-1} ( a^r + b^r) ,    \quad r \ge 1 ,
$$
and 
$$
         2^{r-1} (a^r + b^r) \le (a + b)^r \le  a^r + b^r ,    \quad r \in (0,1). 
$$

\section{Existence and boundedness}\label{sct existence}

In this section, we establish the existence and boundedness of minimizers of \eqref{main functional}. 

\begin{theorem}[Existence]\label{thm:existence}
For a boundary data $g$ satisfying \eqref{boundary data assumptions}, there exists a nonnegative minimizer
\[
u \in H_g^1(B_1)
\]
of the functional \eqref{main functional}.
\end{theorem}

\begin{proof}
    Combining \eqref{fu_gamma_estimate}, or \eqref{f-e-2} with ellipticity we have
\begin{eqnarray}
    J(v) &\ge& \int\limits_{B_1} \frac{\lambda}{2} | \nabla v|^{2}   +  f v^{\gamma} dx \nonumber \\
    &\ge&  \int\limits_{B_1} \frac{\lambda}{2} | \nabla v|^{2}   -  |f| \left( |v - g|^{\gamma}  + |g|^{\gamma} \right) dx \nonumber \\
    &\ge&  \int\limits_{B_1} \frac{\lambda}{2} | \nabla v|^{2} dx  -  C\|f\|_{L^{p,\infty}(B_1)}  \left( \|v - g\|_{H^1(B_1)}^{\gamma}  + \|g\|_{H^1(B_1)}^{\gamma} \right)  \nonumber 
    \end{eqnarray}
    Then by the standard Poincare inequality we have
    \[J(v)\ge  \frac{\lambda}{2} \int\limits_{B_1} | \nabla v|^{2} dx  - C_1\left( \int\limits_{B_1} | \nabla v|^{2} \right)^{\gamma/2} - C_2. \] 
where $C_1 = C(n,  \gamma,  \|f \|_{L^{p,\infty} (B_1)} ) $ and $C_2 = C( \|f \|_{L^{p,\infty} (B_1)} ,  g)$. 
Applying Young's inequality with $\varepsilon > 0$ to the last integral 
in the previous expression, we obtain:
 \[  J(v)
   = (\frac{\lambda}{2} - \varepsilon) \int_{B_1} |\nabla v_j |^2 dx  -  C(n, \gamma, \varepsilon, \| f\|_{L^{p,\infty} (B_1)}, g ). \]

We have verified that the functional $J$ is coercive on $H_g^1(B_1)$. In particular,
\[
\inf_{v \in H_g^1(B_1)} J(v) > -\infty.
\]
Let $$ m: = \inf\limits_{v \in \mathbb{K}} J(v),$$ and $\{v_j\}_{j=1}^{\infty} \in  \mathbb{K}= H_g^1 ( B_1)$, a minimizing sequence. There exists $j_0 \ge 1$ such that 
    $$
        m \le J(v_j) < m + 1, \quad \forall \, j \ge j_0 .
    $$
Moreover,  arguing as before
\begin{eqnarray}
   m +1 > J(v_j) &\ge&  (\frac{\lambda}{2} - \varepsilon ) \int_{B_1} |\nabla v_j |^2 dx  -  C(n, \gamma, \varepsilon, \| f\|_{L^{p,\infty} (B_1)}, g ). \nonumber
\end{eqnarray}
Therefore, by taking $\varepsilon > 0$ small enough, we may conclude 
 that $\{v_j\}_{j=1}^{\infty} $ is a bounded sequence in $H_g^1 ( B_1)$. 
 
Thus, by the Rellich--Kondrachov  theorem, there exists a function $v \in H_g^1 (B_1)$, such that, passing to   subsequence if necessary, 
\begin{eqnarray}
    v_j \to v & &   \; \text{ weakly in } H^1(B), \nonumber \\
v_j \to v & & \;  \text{strongly in } L^q(B) \text{ for all } q<2^*, \nonumber \\
v_j \to v & & \; \text{a.e. in } B_1. \nonumber
\end{eqnarray}

By the weak lower semicontinuity, 
$$
    \int\limits_{B_1} \frac{1}{2} \langle A(x) \nabla v , \nabla v \rangle dx \le \liminf\limits_{j \to \infty} \int\limits_{B_1} \frac{1}{2} \langle A(x) \nabla v_j , \nabla v_j \rangle dx. 
$$
For the second term, let $ 1< r < p$ be such that
\[
\gamma r' < 2^*, \quad \text{where } \frac{1}{r} + \frac{1}{r'} = 1.
\]

By Hölder's inequality, for any measurable set $E \subset B_1$ we have
\[
\int_E |f|\, |v_j|^\gamma \, dx \le 
\|f\|_{L^r(E)} \, \|v_j^\gamma\|_{L^{r'}(B_1)} 
= \|f\|_{L^r(E)} \, \|v_j\|_{L^{\gamma r'}(B_1)}^\gamma.
\]

Since $\gamma r' < 2^*$, the sequence $\{v_j\}_{j=1}^{\infty}$ is bounded in $L^{\gamma r'}(B_1)$, and so there exists $C>0$ independent of $j$ such that
\[
\int_E |f|\, |v_j|^\gamma \, dx \le C \|f\|_{L^r(E)}.
\]

By Lemma \ref{lemma:weak_embedding}, it follows that $\|fv_j^{\gamma}\|_{L^1(E)} \to 0$ as $|E| \to 0$, uniformly in $j$.  
This shows that the sequence $\{f\, v_j^\gamma\}_{j=1}^{\infty} $ is \emph{uniformly integrable} in $L^1(B_1)$.

Moreover, since $v_j \to v$ a.e., we have
\[
f\, v_j^\gamma \to f\, v^\gamma \quad \text{a.e. in } B.
\]

Hence, by Vitali's convergence theorem, we conclude that
\[
\int_B f\, v_j^\gamma \, dx \longrightarrow \int_B f\, v^\gamma \, dx.
\]

Therefore, 
$$
    J(v) \le \liminf\limits_{j \to \infty} J(v_j) = m,
$$
and hence $v$ is a minimum. 
Moreover, since $g\ge 0$ and  $J(\max(u,0)) \le J(u)$, with equality if and only if $0 \le u$, can always be taken to be nonnegative.

\end{proof}


\begin{theorem}[Boundedness of minimizers]\label{thm:boundedness}
Let $u \in H_g^1(B_1)$ be a nonnegative minimizer of \eqref{main functional}, with $g$ as in \eqref{boundary data assumptions}. Then $u$ is uniformly bounded in $B_1$:
\begin{equation}\label{boundedness}
0 \le u \le C \quad \text{a.e. in } B_1,
\end{equation}
where the constant $C>0$ depends only on the dimension, the ellipticity constants,
$\|f\|_{L^{p,\infty}(B_1)}$, and $\|g\|_{L^{\infty}(\partial B_1)}$.
\end{theorem}

\begin{proof}

For the upper bound in \eqref{boundedness}, let $u$ be a  minimizer of $J$ over $H_g^1(B_1)$. Let   $M_0 = \| g \|_{L^{\infty} (\partial B_1)}$ and, for every $M \ge M_0$, define $u_M = \min( u, M )$.  Since \( u = g \) on \( \partial B_1 \), we have \( u_M \in H_g^1(B_1) \). Denote by 
$$
    E_M : = \{ x \in B_1 \suchthat u(x) > M \} . 
$$
It follows from the definition of $u_M$ that 
$$
    u_M = u  \;  \text{in} \; E_M^{c} \quad \text{and} \quad u_M = M \; \text{in} \; E_M. 
$$
By the minimality of $u$ we estimate
\begin{eqnarray}
    0 &\le& J(u_M) - J(u) \nonumber \\
    &\le&  \int\limits_{E_M} - \frac{1}{2} \langle A \nabla u, \nabla u \rangle  +  f(x) [ (M )^{\gamma} - u^{\gamma}] dx. \nonumber 
\end{eqnarray}
Hence, 
$$
    \frac{\lambda}{2}\int\limits_{E_M} | \nabla u|^2  dx  \le  \int\limits_{E_M} \frac{1}{2} \langle A \nabla u, \nabla u \rangle  dx  \le  \int\limits_{E_M}   |f|   | u - M |^{\gamma} dx ,
$$
since $|t^{\gamma} - s^{\gamma} | \le |t -s|^{\gamma}$ , for $\gamma \in (0,1)$  and $t,s \ge 0$.  
Next,  applying \eqref{fu_gamma_estimate} followed by Poincar\'e inequality to $(u -M)_+  \in H_0^1(E_m)$,  yields
$$
     \int\limits_{E_M}   |f|   | u - M |^{\gamma} dx \le C \| f\|_{L^{p,\infty}(E_M)} \| \nabla (u - M )\|_{L^2 (E_M)} ^{\gamma} |E_M|^{\frac{2-\gamma}{2} + \frac{\gamma}{n} - \frac{1}{p}}
$$
Thus, 
\begin{equation}\label{decay for E_M}
     \int\limits_{E_M} |\nabla u|^2 dx \le C |E_M|^{\frac{2}{2-\gamma} \left( \frac{2-\gamma}{2} + \frac{\gamma}{n} - \frac{1}{p} \right)} = C|E_M|^{1 - \frac{2}{n} + \delta}
\end{equation}
where 
$$
    \delta = \frac{2}{n} + \frac{2}{2-\gamma} ( \frac{\gamma}{n} - \frac{1}{p} )
$$
is a positive number, since $p > \frac{n}{2}. $  Therefore,  by  \cite[Chapter~2, Lemma~5.3]{Ladyzhenskaya-Ural'tseva}, 
$$
    \| u \|_{L^{\infty} (B_1)} \le \tilde{C}
$$
for a constant $\tilde{C} > 0$ depending only on dimension , $\delta$, $M_0$, the constant $C$ from estimate \eqref{decay for E_M} and $\| u\|_{L^1 (E_{M_0})}$. 
We now complete the proof by deriving an $L^1$ estimate on $E_{M_0}$. 
Since $(u-M_0)_+ \in H_0^1(B_1)$, by Poincar\'e's inequality and \eqref{decay for E_M}, we obtain
\begin{eqnarray}
    \int_{E_{M_0}} |u| \; dx &\le& \int_{B_1} (u - M_0)_+ \;dx + M_0 | E_{M_0}| \nonumber \\
    &\le&  |B_1|^{\frac{1}{2}} \|(u- M_0)_+\|_{L^2(B_1)} + M_0 | E_{M_0}| \nonumber \\ 
    &\le& c(n) |B_1|^{\frac{1}{n} }\|\nabla (u - M_0)_+ \|_{L^2(B_1)} + M_0 | E_{M_0}| \nonumber \\ 
    &\le& c(n) |B_1|^{\frac{1}{n} }\|\nabla u\|_{L^2(E_{M_0})} + M_0 | E_{M_0}| \nonumber \\ 
    &\le& C 
\end{eqnarray}
where $C > 0$ depends only  on the dimension, the ellipticity constants,
$\|f\|_{L^{p,\infty}(B_1)}$, and $\|g\|_{L^{\infty}(\partial B_1)}$.

\end{proof}

\begin{remark}[Minimal assumptions on $g$]
\label{rem:weak-g}
We finish this section by remarking that the assumptions $g \in C(\overline{B_1})$, $g \ge 0$, $g \not\equiv 0$
were imposed above for simplicity. 
In fact, the results remain valid under weaker assumptions on $g$:

\begin{itemize}
    \item For existence of a minimizer, it is sufficient that
    $g \in H^{1/2}(\partial B_1)$ and $g \ge 0$ (in the trace sense).
    \item For the $L^\infty$ boundedness of minimizers, it is enough that
    $g \in L^\infty(\partial B_1)$ and $g \ge 0$.
\end{itemize}

\end{remark} 

\section{Universal H\"older estimates} \label{sct local reg}

In this section, we establish local H\"older continuity for minimizers of $J_{A,f}$. 
The main strategy relies on a multiscale analysis combining approximation and iterative oscillation control.

Since our interests are local, it will be convenient to work with local minimizers of \eqref{main functional}. We start this section by recalling the definition of local minimum . 

\begin{definition}(\cite[Definition 1.4]{AltPhilips86})
    A nonnegative $u \in H^1 (B_1)$ is a local minimum relative to a subdomain $D \subset B_1$, if $J(u) \le J(u+v)$ for all $v \in H_0^1 (D)$. 
\end{definition}

 Note that if $u$ is a minimizer of $J$ over $H_g^1(B_1)$, then $u$ is a local minimizer of $J$ relative to any subdomain $D \subset B_1$.


\begin{lem}\label{Caccioppoli style}
    Let $u$ be a local minimizer of $J_{A, f}$. Let $h$ be the unique weak solution to
    $$
         \left\{\begin{matrix}
\D (A(x) \nabla h) = 0 & \text{in} &B_{R}(x_0)\\
h = u &\text{in} & B_1 \setminus B_{R} (x_0)
\end{matrix}\right.
    $$
    where $x_0 \in B_{1/2}$, and $0 < R < 1/2$. Then there exists a constant $C > 0$, only depending upon $n, p, \gamma, \lambda, \Lambda$ such that 
    \begin{equation}
        \fint\limits_{B_{R}(x_0)} | u - h |^2 dx \le C \|f \|_{L^{p,\infty}(B_1)}^{\frac{2}{2-\gamma}} R^{\sigma_0} 
    \end{equation}
    where $\sigma_0=\frac{2(2-n/p)}{2-\gamma}$ if $n\ge 3$ and $\sigma_0=\frac{2(2-2/p-\epsilon)}{2-\gamma}$ if $n=2$.
\end{lem}
\begin{proof}
    Since 
    $$
        J(u, B_{R}(x_0)) \le J(h, B_{R}(x_0)),
    $$
    we have
    \begin{equation}
        \int\limits_{B_{R}} \langle A(x) \nabla u, \nabla u \rangle  - \langle A(x) \nabla h, \nabla h \rangle dx  \le  \int\limits_{B_{R}} f(x) [ h^{\gamma} - u^{\gamma}] dx . 
    \end{equation}
   From the PDE satisfied by $h$, 
    \begin{align}\label{estimate LHS Grad}
        \frac{\lambda}{2} \int\limits_{B_{R}} |\nabla (u - h) |^2 dx &\le \frac{1}{2} \int\limits_{B_R} \langle A \nabla (u -h) , \nabla (u - h) \rangle dx  \\
        &=\frac 12\int_{B_R} \langle A(x)\nabla u,\nabla (u-h)\rangle \nonumber \\
        &=\frac 12\int_{B_R}\langle A\nabla u,\nabla u\rangle-\frac 12\int_{B_R}\langle A\nabla u, \nabla h\rangle \nonumber\\
        &=\frac 12\int_{B_R} \langle A\nabla u,\nabla u \rangle -\frac 12 \int_{B_R} \langle A(\nabla u-\nabla h+\nabla h,\nabla h \rangle \nonumber \\
        &=\frac{1}{2}  \int\limits_{B_{R}} \langle A(x) \nabla u, \nabla u \rangle  - \langle A(x) \nabla h, \nabla h \rangle dx .\nonumber
    \end{align}
  Moreover, since $p > \frac{n}{2}$, we set $r$ to be $\frac 1r=1-\frac{(n-2)\gamma}{2n}$ if $n\ge 3$ and $\frac 1r=1-\epsilon$ if $n=2$. Then applying  \eqref{fu_gamma_estimate} and \eqref{f-e-2}) we have 
    \begin{eqnarray} \label{estimate RHS}
         \int\limits_{B_{R}} f(x) [ h^{\gamma} - u^{\gamma}] dx &\le&  \int\limits_{B_{R}} |f(x)| \,  | h - u|^{\gamma} dx \nonumber \\
        &\le&c(p,n,\gamma)  \, \|f \|_{L^{p,\infty} (B_R)}  \| u  - h \|_{H^1(B_R)}^{\gamma} \,|B_R|^{\frac{1}{r} - \frac{1}{p}} \nonumber \\
        &\le& c(p,n,\gamma)  \, \|f \|_{L^{p,\infty} (B_R)}  \| \nabla (u  - h) \|_{L^2(B_R)}^{\gamma} \,|B_R|^{\frac{1}{r} - \frac{1}{p}}.
    \end{eqnarray}
   where the last inequality follows from the Poincar\'e inequality.
    Putting \eqref{estimate LHS Grad} and \eqref{estimate RHS} together we get
    
   \begin{equation*}
        \frac{\lambda}{2} \int\limits_{B_{R}} |\nabla (u - h) |^2 dx  \le c(p, n, \gamma)\|f \|_{L^{p,\infty} (B_R)}  \|  \nabla( u - h)  \|_{L^2 (B_R)}^{\gamma}  |B_R |^{\frac{1}{r} - \frac{1}{p}} .
   \end{equation*}
   Therefore, 
    \begin{equation}\label{ineq with gradient}
    \| \nabla (u - h) \|_{L^2 (B_R)}^{2 - \gamma} \le c(\lambda, n,  p,\gamma) \|f \|_{L^{p,\infty} (B_R)}  |B_R |^{\frac{1}{r} - \frac{1}{p}} .
   \end{equation}
    Next, using $p>n/2$ and $R<1$, we combine \eqref{ineq with gradient} and the Poincar\'e inequality once more to obtain
    \begin{eqnarray}\label{using poincare}
        \int_{B_R} |u - h|^2 dx &\le& c(n) R^2 \int_{B_R} |\nabla (u - h) |^2 dx \nonumber \\
        &\le&  C R^2\;  \|f \|_{L^{p,\infty} (B_R)}^{\frac{2}{2-\gamma}}   |B_R |^{\frac{2}{2-\gamma}( \frac{1}{r}- \frac{1}{p})} \nonumber\\
        &=& C \|f \|_{L^{p,\infty} (B_R)}^{\frac{2}{2-\gamma}}   R^{2 + \frac{2n}{2-\gamma} (\frac{1}{r} -\frac{1}{p}) } . \nonumber \nonumber \\
        &\le& C\|f \|_{L^{p,\infty} (B_R)}^{\frac{2}{2-\gamma}} R^{n+\sigma_0}.\nonumber
    \end{eqnarray}
    
\end{proof}

 Before proceeding, we recall that solutions the  homogeneous equation $\D(A(x) \nabla h ) = 0$, are  H\"older continuous with estimates
\begin{equation}\label{continuity of h}
    \| h\|_{C^{0,\alpha_0} (B_{1/4})} \le C_*\| h\|_{L^2 (B_{1/2})} ,
\end{equation}
for some $\alpha_0 \in (0,1)$ and a constant $C_* > 0$ depending only on $n, \lambda, \Lambda$. 

This estimate provides a first-scale control of minimizers of $J_{A,f}$.

\begin{prop}\label{choice of constants prop}
     There exist universal constants $\varepsilon_0 > 0$, $0 < \rho < 1/4$, depending only on $n, \lambda, \Lambda$, such that: if $u$ be a local minimizer of $J_{A,f}$, with $\fint_{B_1} u^2 dx \le 1$, $A \in \mathcal{M}_{\lambda, \Lambda} (B_1)$, and  $\| f \|_{L^{p,\infty} (B_1)} \le \varepsilon_0$,  then 
    \begin{equation}
        \fint_{B_{\rho}} |u - \mu_1| ^2 dx \le \rho^{\alpha}, 
    \end{equation}
    for some universally bounded constant $|\mu_1 | \le C_1(n, \lambda, \Lambda)$, and any 
    $$
        0 < \alpha < \min \left\{ 2 \alpha_0 , 2 \theta \right\} 
    $$
\end{prop}
\begin{proof}
    Let $h$ be the harmonic replacement of $u$ in $B_{1/2}$. By Lemma \ref{Caccioppoli style},
    $$
        \int\limits_{B_{1/2}} |u - h |^{2} dx \le  C \|f \|_{L^{p,\infty} (B_1)}  ^{\frac{2}{2 - \gamma}}. 
    $$
    where $C >0$ depends only on $n,\lambda,\Lambda$, $p$ and $\gamma$.  
    Thus, 
    \begin{eqnarray}\label{to estimate h in L^2}
        \| h\|_{L^2 (B_{1/2})} &\le&  \| u - h \|_{L^2 (B_{1/2})} + \| u \|_{L^2 (B_{1/2})} \nonumber \\ 
        &\le&  C \|f \|_{L^{p,\infty} (B_1)}  ^{\frac{1}{2 - \gamma}} + \| u \|_{L^2 (B_{1})}  \nonumber \\
        &\le& C\varepsilon_0^{1/(2 - \gamma)}  + \sqrt{|B_{1}|}  \nonumber \\ 
        &\le& (C +1) \sqrt{|B_{1}|} 
    \end{eqnarray}
    if we choose $\varepsilon_0 \le \left( \sqrt{|B_{1}|} \right)^{2 - \gamma} $ . 
Now set \(\mu_1 := h(0)\). 
By Lemma~\ref{Caccioppoli style}, the Hölder continuity of \(h\) (see \eqref{continuity of h}), and \eqref{to estimate h in L^2}, 
we obtain, for \(\rho < \tfrac{1}{4}\) to be chosen later,
\begin{align*}
\fint_{B_{\rho}} |u - \mu_1|^2 \, dx 
&\le 2 \left( \fint_{B_{\rho}} |u - h|^2 \, dx + \fint_{B_{\rho}} |h - \mu_1|^2 \, dx \right) \\
&\le 2 \left( \fint_{B_{\rho}} |u - h|^2 \, dx +C_* \| h \|_{C^{0,\alpha_0}(B_{1/2})}^2 \rho^{2\alpha_0} \right) \\
&\le 2 \left( \fint_{B_{\rho}} |u - h|^2 \, dx + C_* \| h \|_{L^2(B_{1/2})}^2 \rho^{2\alpha_0} \right) \\
&\le 2 \left( C\, \varepsilon^{\frac{2}{2-\gamma}} \rho^{\sigma_0} 
+ \tilde{C_*} (C + 1) \rho^{2\alpha_0} \right),
\end{align*}
for all \( \rho \in (0, \tfrac{1}{4}) \).
Next we choose $0 < \rho < 1/4$ such that 
    $$
         C_{*} (C + 1) \rho^{2\alpha_0} \le \frac{1}{4} \rho^{\alpha}
    $$
   for $\alpha < 2 \alpha_0 $. In  the sequel, we choose $\varepsilon_0$ even smaller such that 
    $$
        C \varepsilon_0^{\frac{2}{2-\gamma}} \rho^{\sigma_0} \le  C \varepsilon_0^{\frac{2}{2-\gamma}} \le  \frac{1}{4} \rho^{\alpha} .
    $$
    Hence,
    $$
         \fint_{B_{\rho}} |u - \mu_1| ^2 dx \le \rho^{\alpha} .
    $$
    
    
\end{proof}

In what follows, we iterate the previous estimate to establish the continuity of $u$ under a smallness condition on the source term $f$.

\begin{prop}\label{reg under smallness}
    There exist constants $\varepsilon_0 \in (0,1)$ and  $\rho \in (0, \frac{1}{4} )$, depending only on $n$,  such that:  for $\varepsilon \in (0, \varepsilon_0)$ and $ u$ a local minimizer of $J_{A,f}$,  with $\displaystyle \fint_{B_1} u^2  dx \leq 1$, $A \in \mathcal{M}_{\lambda, \Lambda}(B_1)$, and $\| f\|_{L^{p,\infty} (B_1)} \le \varepsilon $, there holds 
    \begin{equation}
        |u (x) - u (0) | \le C |x|^{\beta}, \quad \forall |x| \,  \le \rho
    \end{equation}
     for some   $\beta \in (0,1)$, $C >0$,  universal constants, depending additionally of $\varepsilon_0$, $\rho$. 
\end{prop}

\begin{proof}
    Let $\varepsilon_0 > 0$ and $\rho > 0$ be the constants given by Proposition \ref{choice of constants prop}. Define  $u_1 : B_1 \to \mathbb{R}$,  by 
    $$
        u_1 (x) = \frac{u(\rho x) - \mu_1}{\rho^{\alpha/2}}.
    $$
    Then, by the previous proposition $\fint_{B_1} u_1^2 dx \le 1$, and one can check that $u_1$ is a local minimizer of 
    $$
        J_1(v, B_1) = \int\limits_{B_1} \left\{ \frac{1}{2}\langle\Tilde{A} (x) \nabla v, \nabla v\rangle  +   \Tilde{f}(x) (v + \rho^{-\alpha/2}\mu_1 )_+^{\gamma} \right\} dx
    $$
    where $\Tilde{A}(x) = A(\rho x)$ has the same ellipticity constants as $A$,  and 
    \[\Tilde{f}(x) =  \rho^{2 -\alpha(1 - \gamma/2)} f(\rho x).\]  Indeed, for any $\varphi \in H_0^1(B_1)$, with sufficiently small norm, we define 
    \begin{eqnarray}
        v(y) &=& \mu_1 + \rho^{\alpha /2 } (u_1 (y/\rho ) + \varphi (y/\rho)),  \;\;\; y =\rho x\nonumber \\
        &=& u(y) + \rho^{\alpha /2} \varphi (y/\rho).
    \end{eqnarray}
    Then, $v -u \in H_0^1 (B_{\rho})$ and by the local minimality of $u$, 
    $$
        J(u, B_{\rho}) \le J(v, B_{\rho} ). 
    $$
    After a change of variables, 
    $$
         J(u, B_{\rho})  = \rho^{n + \alpha -2} J_1 (u_1, B_1), \; \; \text{and} \;\; J(v, B_{\rho}) = \rho^{n +\alpha -2} J_1 (u_1 + \varphi , B_1).
    $$
    Since $\tilde{f}(x)$ satisfies
    $$
        \| \tilde{f} \|_{L^{p,\infty}(B_1)} \le \rho^{2 -\alpha(1 - \gamma/2) - n/p} \| f\|_{L^{p,\infty} (B_1)} \le \| f\|_{L^{p,\infty} (B_1)} ,
    $$
    as long as  $\alpha < \frac{2 - n/p}{1 - \gamma/2}$, we can  apply  Proposition \ref{choice of constants prop} to find another constant $\tilde{\mu}$, with $|\tilde{\mu}| \le C_1(n, \lambda, \Lambda)$, such that 
    $$
        \fint_{B_{\rho}} |u_1 - \tilde{\mu}|^2 dx \le \rho^{\alpha},
    $$
    which translates back to $u$ as
    $$
        \fint_{B_{\rho^2}} |u - \mu_2|^2 dx \le \rho^{2\alpha}.
    $$
    where $\mu_2 = \mu_1 + \rho^{\alpha/2} \tilde{\mu}$.  Proceeding with the inductive argument, we conclude that
    $$
        \fint_{B_{\rho^{k}}} |u - \mu_{k}|^2 dx \le \rho^{k\alpha}.
    $$
    where $\mu_{k} = \mu_{k-1} +  \rho^{(k-1)\frac{\alpha}{2}} \tilde{\mu} _{k-1}$. It follows from universal bounds that
    $$
        |\mu_{k} - \mu_{k-1} | \le C_1\rho^{(k-1)\frac{\alpha}{2}}.
    $$
    Thus the above sequence is convergent.  Let $\mu_{\infty} = \lim\limits_{k \to \infty} \mu_k$. We estimate
    $$
        |\mu_{\infty} - \mu_k| \le \lim\limits_{N \to \infty} C_1 \sum\limits_{i = k}^N \rho^{i(\alpha/2)} \le \frac{C_1}{1 - \rho^{\alpha/2}} \rho^{k(\alpha/2)} .
    $$
    We conclude then 
    \begin{eqnarray}
         \fint_{B_{\rho^{k}}} |u - \mu_{\infty}|^2 dx &\le & 2 \left(  \fint_{B_{\rho^{k}}} |u - \mu_{k}|^2 dx  +  \fint_{B_{\rho^{k}}} |u_{k} - \mu_{\infty}|^2 dx \right) \nonumber \\
         &\le& \left[ 1+ \left( \frac{C_1}{1 - \rho^{\alpha/2}}\right)^2 \right] \rho^{k\alpha}
    \end{eqnarray}
    Now, let $u_B = \fint_B u dx$. Then 
    \begin{eqnarray}
        \fint_{B_{\rho^{k}}} |u - u_{B_{\rho^k}} |^2 dx \le 2\left( \fint_{B_{\rho^{k}}} |u - \mu_{\infty} |^2 dx + \fint_{B_{\rho^{k}}} |\mu_{\infty} - u_{B_{\rho^k}} |^2 dx\right) \nonumber
    \end{eqnarray}
    Moreover, 
   \begin{eqnarray}
       \fint_{B_{\rho^{k}}} |\mu_{\infty} - u_{B_{\rho^k}} |^2 dx &=& \fint_{B_{\rho^k}} \left| \fint_{B_{\rho^k}} u - \mu_{\infty} dx \right|^2 dx  \nonumber \\
       &\le&  \fint_{B_{\rho^k}} |u - \mu_{\infty} |^2 dx   \nonumber
   \end{eqnarray}
    by Jensen's inequality. Finally, 
    $$
         \fint_{B_{\rho^{k}}} |u - u_{B_{\rho^k}} |^2 dx \le 4 \left[ 1+ \left( \frac{C_1}{1 - \rho^{\alpha/2}}\right)^2 \right] \rho^{k\alpha} . 
    $$
    Thus, applying Campanato’s theorem, \cite[Theorem I.2]{Campanato1963} , we conclude that $u$ is H\"older continuous at the origin.  
\end{proof}

We are now ready to prove the main result of this section.

\begin{proof}[Proof of Theorem \ref{Holder reg}]
    Let $\varepsilon_0$ be the number given by Proposition \ref{reg under smallness}. Fix  $\varepsilon \in (0, \varepsilon_0)$ and $x_0 \in B_{1/2}$.  Define 
$$
    \tilde u(y) = \kappa u(x_0 + ry)
$$
where 
$$
    r = \min \left( \frac{dist(x_0, \partial B_1)}{10}, \left( \frac{\varepsilon_0}{\kappa^{2-\gamma} \| f\|_{L^p (B_1)}}\right)^{p/(2p-n)}  \right) \,\,\, \text{and} \, \,\, \kappa =  \left(\frac{1}{\fint_{B_r (x_0)} u^2 dx} \right)^{1/2} . 
$$
Hence $\tilde{u}$ is a minimizer of $J_{\tilde{A}, \tilde{f}}$, with $\tilde A (y) = A(x_0 + r y)$, and  
$$
    \| \tilde f\|_{L^{p,\infty} (B_1)} \le k^{2 -\gamma} r^{2 - n/p} \|f\|_{L^{p,\infty} (B_1)} \le \varepsilon_0.
$$ 
Moreover, 
$$
    \fint_{B_1} \tilde u^2 \,  dx \le 1.
$$
Therefore,  by Proposition \ref{reg under smallness}
$$
    |\tilde u (y) - \tilde u (0) | \le C |y|^{\beta}, 
$$
whenever $|y| < \rho$.  Translating this estimate back to $u$ yields, 
$$
    |u(x_0 + ry) - u(x_0) | \le C \| u\|_{L^2(B_1)} |y|^{\beta} .
$$
This implies that $u$ is H\"older continuous at $x_0$. 
\end{proof}

As a consequence of the continuity of $u$, it follows that $  \{ u > 0\}$ is an open set, and we have the following 
\begin{corollary}
Let $u$ be a minimizer of 
\[
J(v) = \int_B \Bigl( \tfrac12 \langle A\nabla v, \nabla v \rangle + f(x) v_+^\gamma \Bigr)\, dx
\]
over $v \in H_g^1(B)$. Then  $u$ satisfies the Euler-Lagrange equation
\begin{equation}\label{the euler-lagrange equation}
\operatorname{div}(A(x) \nabla u) = \gamma f(x) u^{\gamma-1} 
\quad \text{in } B \cap \{ u>0 \},
\end{equation}
in the weak sense.
\end{corollary}

\begin{proof}
Since $u$ is continuous, $\{u>0\}$ is an open set. Let $U \subset B$ be an open subset where $u>0$ a.e., and let $\varphi \in C_0^\infty(U)$. Because $u>0$ on the support of $\varphi$, there exists $\varepsilon>0$ such that for all $|t|<\varepsilon$, the variation
\[
u_t := u + t\varphi
\]
remains admissible (in particular, $u_t \ge 0$). 

By minimality of $u$, the map $t \mapsto J(u+t\varphi)$ has a critical point at $t=0$, yielding
\[
0 = \frac{d}{dt} J(u+t\varphi)\Big|_{t=0} 
= \int_B \bigl( \langle A\nabla u, \nabla \varphi \rangle + \gamma f(x) u^{\gamma-1} \varphi \bigr) dx.
\]

Since $\varphi$ is supported in $U$, the integral reduces to $U$:
\[
\int_U \langle A\nabla u, \nabla \varphi \rangle dx + \int_U \gamma f(x) u^{\gamma-1} \varphi dx = 0.
\]

This is precisely the weak formulation of \eqref{the euler-lagrange equation} .
\end{proof}

\section{Regularity along the Free Boundary}\label{sct sharp reg}

In this section, our main goal is to describe the precise geometric behavior of a local minimizer near free boundary points. Roughly speaking, (weak) solutions to \eqref{main equation} should adjust their vanishing  rate according to the singularity of the source term.

We will need the following Caccioppoli-type estimate:
\begin{lem}\label{caccioppoli}
    There is a constant $C> 0$ depending only on $n, \lambda, \Lambda, p, \gamma$, such that any minimizer of $J_{A,f}$ over $H_g^1 (B_1)$ satisfies 
    \begin{equation}\label{cacciop estimate}
        \int\limits_{B_1} |\nabla u|^2 dx \le C \left( \| u \|_{L^2 (B_1)}^{2} + \|f\|_{L^{p,\infty} (B_1)}^{\frac{2}{2-\gamma}} \right). 
    \end{equation}
\end{lem}

\begin{proof}
    Let $\xi \in C_0^{\infty} (B_1)$ such that $\xi \equiv 1$ in $B_{1/2}$, $\xi \equiv 0 $ in $B_{1} \setminus B_{3/4}$ and $|\nabla \xi | \le C$. Thus $w = u (1 -\xi^2 ) \in H_{g}^1 (B_1)$ and by the minimality of $u$ we have 
    $$
        J(u) \le J(w),
    $$
    which implies: 
    \begin{eqnarray}
        \frac{1}{2} \int_{B_1} (2\xi - \xi^2) \langle A \nabla u, \nabla u\rangle dx &\le& \int_{B_1} 2  u \xi (1- \xi^2) \langle  \nabla u, A \nabla \xi \rangle  - 2  u^2 \langle A \nabla \xi , \nabla \xi \rangle dx  \nonumber \\ 
        && + \int_{B_1} |f| |  \xi^2 u|^\gamma dx \nonumber
    \end{eqnarray}
    By  Cauchy's inequality, 
    \begin{eqnarray}
      \int_{B_1} 2  u \xi (1- \xi^2) \langle  \nabla u, A \nabla \xi \rangle dx  &\le& \int_{B_1} \frac{\lambda}{4} \xi^2 |\nabla u|^2 dx + \frac{4}{\lambda} u^2 |A \nabla \xi|^2 dx. \nonumber
    \end{eqnarray}
     Therefore, by ellipticity 
    \begin{eqnarray}
         \int_{B_1} \frac{\lambda}{2} \xi^2 | \nabla u|^2 dx &\le&  \int_{B_1} \frac{\lambda}{2} (2\xi^2 - \xi^4) |\nabla u|^2 dx \nonumber \\
        &\le& \int_{B_1}  (2\xi - \xi^2) \langle A \nabla u, \nabla u\rangle dx \nonumber\\
        &\le&  \int_{B_1} \frac{\lambda}{4} \xi^2 |\nabla u|^2 dx + \frac{4}{\lambda} u^2 |A \nabla \xi|^2 dx  + \int_{B_1} |f| |  \xi^2 u|^\gamma dx 
    \end{eqnarray}
    
   Next, since $\xi^2 u \in H_0^1 (B_1)$ then by \eqref{fu_gamma_estimate}, we estimate
    \begin{eqnarray}
        \int\limits_{B_1} f |\xi^2  u|^{\gamma} dx &\le& C(n, p, \gamma)\|f\|_{L^{p,\infty} (B_1)} \| \xi^2 u\|_{H^1 (B_1)}^{\gamma} \nonumber \\
        &\le& C \|f\|_{L^{p,\infty} (B_1)} \| \nabla (\xi^2 u)\|_{L^2 (B_1)}^{\gamma} \nonumber \\
        &\le& \frac{C(2-\gamma) \|f\|_{L^{p,\infty} (B_1)}^{2 / (2-\gamma)} }{2\varepsilon} + \frac{\gamma \varepsilon^{2/\gamma}}{2} \| \nabla (\xi^2 u)\|_{L^2(B_1)}^2 \nonumber \\
        &\le& C \left(  \int_{B_1} u^2 dx + \|f\|_{L^{p,\infty} (B_1)}^{\frac{2}{2-\gamma}} \right) + \frac{\gamma \varepsilon^{2/\gamma}}{2}  \int_{B_1}\xi^2 |\nabla u|^2 dx  \nonumber \\
    \end{eqnarray}
    for any $\varepsilon > 0$. Finally, we obtain

    \begin{eqnarray}
         \int_{B_1} \left( \frac{\lambda}{2} - \frac{\lambda}{4} - \frac{\gamma  \varepsilon^{2/\gamma}}{2} \right) \xi^2 |\nabla u|^2 dx \le C \left( \| u\|_{L^2 (B_1)}^2 + \|f\|_{L^{p,\infty} (B_1)}^{2 / (2-\gamma)} \right) \nonumber
    \end{eqnarray}
    
   Since $\xi = 1$ in $B_{1/2}$ then taking $\varepsilon > 0$ so that $\left( \frac{\lambda}{2} - \frac{\lambda}{4} - \frac{\gamma  \varepsilon^{2/\gamma}}{2} \right)> 0 $, the result follows. 
    
\end{proof} 

We now present a key approximation result in the small-data regime.

\begin{lem}[Approximation]\label{Approx lemma}
    Given $\delta > 0$ there exists a $\eta > 0$ depending only on $\delta$, dimension, ellipticity, p, and  $\gamma$, such that if $u$ is a local minimizer of $J_{A,f}$ over $H_g^1 (B_1)$, $g \ge 0$, with $A \in \mathcal{M}_{\lambda, \Lambda} (B_1)$, and $\| f\|_{L^{p,\infty}(B_1)} \le \eta$ , $\fint_{B_1} u^2 dx \le 1$, $u (0) = 0$ then 
    $$
        \sup\limits_{B_{1/4}} u_0 \le \delta.
    $$
    
\end{lem}
\begin{proof}
    Suppose by contradiction that there is an $\delta_0 > 0$ and  sequences  $g_k \in H^1(B_1) \cap L^{\infty} (B_1)$, $f_k \in L^{p,\infty}(B_1)$,  $A_k \in \mathcal{M}_{\lambda, \Lambda} (B_1)$, and $u_k \in H_{g_k} ^1 (B_1)$ such that
    \begin{enumerate}
        \item $\fint_{B_1} u_k^2 dx \le 1$, $u_k(0) = 0$;

        \item $\|f_k \|_{L^{p,\infty} (B_1)} \le \frac{1}{k}$;

        \item $u_k$ is a local minimizer of 
        $$
            J_k (v) = \int\limits_{B_1} \left\{ \frac{1}{2} \langle A_k \nabla v, \nabla v\rangle + f_k (x) (v)_+^{\gamma} \right\} dx.
        $$
    \end{enumerate}
        However, 
        \begin{equation}\label{contradicao eq}
            \sup\limits_{B_{1/4}} u_k \ge \delta_0 .
        \end{equation}
        By the universal H\"older estimates obtained in the previous section, $u_k \to u_{\infty}$ locally uniformly in $B_{1/2}$. By the Cacciopoli estimate \ref{cacciop estimate}, $u_k$ is bounded in $H_{g_k}^1 (B_1)$, and thus, up to a sequence,
        $$
            u_k \to u_{\infty}\,\, \text{weakly in }\,\,  H^1 (B_{1/2}), \quad u_k \to u_{\infty} \,\,\,  \text{in} \,\,\,  L^2(B_{1/2}) .
        $$
        Moreover, by ellipticity, the sequence $(A_k)$ is uniformly bounded in $L_{loc}^2(B_{1/2}, \mathbb{R}^{n\times n} ).$ Hence, there exists a subsequence (still denoted $A_k$) such that $A_k \to A_{\infty}$ weakly in $L_{loc}^2 (B_{1/2} , \mathbb{R}^{n \times n} )$ to some $A_{\infty} \in \mathcal{M}_{\lambda, \Lambda} (B_{1/2})$.  

        We want to verify that $u_{\infty}$ is a weak solution to 
        \begin{equation}\label{limiting PDE}
            \D(A_{\infty} \nabla u ) = 0 \quad \text{in} \quad B_{1/2}.
        \end{equation}
        Indeed, for any $ v \in C_0^{\infty} (B_{1/2})$, we have 
        $$
            J_{k}(u_{k}) \le J_k (u_k + v )
        $$
     Thus, 
        \begin{eqnarray}
           0  & \le  & \int\limits_{B_{1/2}} 2\langle A_k \nabla u_k , \nabla v \rangle +  \langle A_k \nabla v, \nabla v \rangle dx + \int\limits_{B_{1/2}} 2f_k \left[ (u_k + v)_+^{\gamma} - (u_k)_+^{\gamma} \right] dx \nonumber \\
            &=& \text{I}_k + \text{II}_k, \nonumber
        \end{eqnarray}
        where 
        $$
            \text{I}_k =  \int\limits_{B_{1/2}} 2\langle A_k \nabla u_k , \nabla v \rangle + \langle A_k \nabla v, \nabla v \rangle dx 
        $$
        and 
        $$
           \text{II}_k =  \int\limits_{B_{1/2}} 2f_k \left[ (u_k + v)_+^{\gamma} - (u_k)_+^{\gamma} \right] dx. 
        $$
        Once again, using \eqref{lemma:weak_embedding}, 
        $$
           |  \text{II}_k |\le C(n, p,\gamma) \| f_k\|_{L^{p,\infty} (B_1)} \| v\|_{L^2 (B_1)}^{\gamma} \to 0 \,\, \text{as} \,\, k \to \infty. 
        $$
        The second term of $\text{I}_k$ converges to $\int\limits_{B_{1/2}} \langle A_{\infty} \nabla v, \nabla v \rangle dx$ by the weak convergency of $A_k$. For the first term of $\text{I}_k$ we have 
        \begin{eqnarray}
            \int\limits_{B_{1/2}} \langle A_k \nabla u_k , \nabla v \rangle dx = \int\limits_{B_{1/2}} \langle A_k ( \nabla u_k - \nabla u_{\infty}) , \nabla v \rangle dx + \int\limits_{B_{1/2}} \langle A_k \nabla u_{\infty} , \nabla v \rangle dx. \nonumber
        \end{eqnarray}
        Therefore, 
       \begin{eqnarray}
             \lim\limits_{k \to \infty}  \int\limits_{B_{1/2}} \langle A_k \nabla u_k , \nabla v \rangle dx &\le&  \lim\limits_{k \to \infty}  \int\limits_{B_{1/2}} \Lambda |\nabla u_{k} - \nabla u_{\infty}| \, |\nabla v| dx   \nonumber \\
             && \quad +  \lim\limits_{k \to \infty} \int\limits_{B_{1/2}} \langle A_k \nabla u_{\infty} , \nabla v \rangle dx \nonumber \\
            &=& 0 \quad +  \quad \int\limits_{B_{1/2}} \langle A_{\infty} \nabla u_{\infty} , \nabla v \rangle dx \nonumber
       \end{eqnarray}
        by the weak convergency of $\nabla u_k$ to $\nabla u_{\infty}$ and $A_k$ to $A_{\infty}$. 
       
 Therefore, 
 $$
    0 \le   \lim\limits_{k \to \infty}  \text{I}_k + \text{II}_k \le  \int\limits_{B_{1/2}} \langle A_{\infty} \nabla u_{\infty} , \nabla v \rangle dx + \frac{1}{2}\int\limits_{B_{1/2}} \langle A_{\infty} \nabla v, \nabla v \rangle dx
 $$
 which implies
 $$
    0 \le \frac{1}{2}\int\limits_{B_{1/2}} \langle A_{\infty} \nabla (u_{\infty} + v ), \nabla (u_{\infty} + v ) \rangle dx  -  \frac{1}{2}\int\limits_{B_{1/2}} \langle A_{\infty} \nabla u_{\infty} , \nabla u_{\infty} \rangle dx ,
 $$
 and therefore, $u_{\infty}$ satisfies \eqref{limiting PDE} in the weak sense.

To conclude, we note that since the trace of $u_{\infty}$ on $\partial B_{1/2}$ is nonnegative, then $u_{\infty} \ge 0$ in $B_{1/2}$. Therefore, since $u_{\infty} ( 0) = 0$, then by Harnack Inequality (or Maximum principle), $u_{\infty} \equiv 0$ in $B_{1/4}$. Thus, contracting \eqref{contradicao eq} for $k \gg 1$, sufficiently large. 
 \end{proof}

In what follows, we show that $u$ is $C^{\theta}$ at the origin, regardless of the regularity of $A(x)$, for the sharp exponent 
$$
    \theta = \frac{2 - n/p}{2 - \gamma} . 
$$

\begin{proof}[Proof of Theorem \ref{improv reg}]
We recall that we are assuming $0$ is a free boundary point,  $u(0) = 0$. For $ 0 < \rho_{*} \le 1/8$ to be chosen later, consider the function $v(x) = \kappa u( \rho_{*} x) $ defined in $B_1$. It is easily checked that $v$ minimizes 
    $$
        \int_{B_{1}} \left\{ \frac{1}{2} \langle \tilde{A} \nabla v, \nabla v \rangle + \tilde{f} (v)_+^{\gamma} \right\}  dx
    $$
where $\tilde{A} (x) = A(\rho x)$, is a $(\lambda, \Lambda)$- uniformly elliptic matrix,  and  
$$
    \tilde{f}(x)= \kappa^{2 - \gamma} \rho^{2} f(\rho x).
$$ 

We select $\delta = 4^{-\theta}$ and consider the correspondent   $\eta$ from Lemma \ref{Approx lemma}. Next, we choose,  $\rho \ll 1$ such that 
\begin{equation}\label{small regime loc min}
    0 < \rho_{*} \le \left( \eta \| f\|_{L^{p,\infty} (B_1)}^{-1} \right)^{\frac{p}{(2p -n )}} \quad \text{and} \quad \kappa^2 \le \min\{ 1, \rho_*^n (\fint_{B_1} u^2 dx )^{-1} \}
\end{equation}

With these choices, $v$ is under the assumptions of Lemma \ref{Approx lemma}. Thus, 
$$
    \sup\limits_{B_{1/4}}  v(x)  \le 4^{-\theta} . 
$$
In the sequel, assume by induction that for some $k \in \mathbb{N}$, 
\begin{equation}\label{geometric decay}
    \sup\limits_{B_{4^{-k}}} v \le \frac{1}{4^{k \theta} }. 
\end{equation}
Let us define $v_k: B_1 \to \mathbb{R}$ by 
$$
    v_k(x) = 4^{k\theta} v \left(\frac{x}{4^k} \right)   .
$$
We verify that $v_k$ minimizes 
$$
     \int_{B_{1}} \left\{ \frac{1}{2} \langle A_k \nabla v, \nabla v \rangle + f_k  (v)_+^{\gamma} \right\}  dx
$$
where $A_k (x) = \tilde{A}(4^{-k} x) $, $f_k(x) = 4^{k[ \theta ( 2-\gamma) - 2 ]} \tilde{f}(4^{-k} x )$. We note that $A_k$ is still a $(\lambda, \Lambda)$-elliptic matrix, and because of the sharp choice of $\theta$, 
$$
    \| f_k \|_{L^{p,\infty}(B_1)} = \| \tilde{f} \|_{L^{p,\infty}(B_1)} \le \eta .  
$$
Moreover, the inductive hypothesis gives $\fint_{B_1} v_k^2 dx \le 1$.  By Lemma \ref{Approx lemma}, 
$$
    \sup\limits_{B_{4^{-1}}} v_k \le \frac{1}{4^{\theta}}
$$
or, 
$$
    \sup\limits_{B_{4^{-{(k+1)}}}} v \le \frac{1}{4^{(k+1) \theta}}. 
$$
This proves \eqref{geometric decay} for all $k \ge 1$. Now, for any $x \in B_{1/4}$, there exists $k \ge 1$ such that
$$
    4^{-(k+1)} \le |x| \le 4^{-1}. 
$$
Thus, 
$$
v(x) \le \sup_{B_{4^{-k}}} v \le 4^{-k\theta} \le 4^{\theta} |x|^{\theta}. 
$$
Translating back to $u$ yields,
$$
    u(y) \le C(\kappa, \rho_*, \theta ) |y|^{\theta}, \,\, \forall y \in B_{\frac{\rho_*}{4}} . 
$$
\end{proof}

\section{Nondegeneracy}\label{sct nondeg}

In this section, we establish Theorem \ref{non-deg theom} and Theorem \ref{non-de-2}.

\begin{proof}[Proof of Theorem \ref{non-deg theom}]
Given $x_0 \in \{ u > 0\}$, let $2r$ be the distance from $x_0$ to the free boundary. Suppose the point that $B_{4r}(x_0)$ touches the free boundary is $y_0$. Trivially $x_0\in B_{4r_0}(y_0)$.
$$
    2r = |x_0  - y_0| = dist( x_0 , \partial\{u > 0\} ).
$$ Note that $u $ solves 
$$
        \D (A(x) \nabla u) = \gamma f(x) u^{\gamma - 1} \quad \text{in} \quad B_{2r}(x_0) ,
    $$
    
 Let $\eta \in C_0^\infty(B_{\frac 32r}(x_0))$ with $\eta \equiv 1$ on $B_r$ and $|\nabla \eta| \lesssim r^{-1}$.
    Multiplying the equation for $u$ by  $\eta^2 u^{1-\gamma}$ and integrating, we have:
    \[
        \int_{B_{2r}(x_0)} \D (A(x) \nabla u) \,\eta^2 u^{1-\gamma} \,dx = \int_{B_{2r}(x_0)} \gamma f(x) \, \eta^2 \,dx.
    \]
   Clearly integration by parts yields the energy inequality:
    \begin{align*}
         &\gamma \int_{B_{2r}(x_0)}  f(x) \eta^2 dx  \\
         =& - \int_{B_{2r}(x_0)} \langle A\nabla u , \nabla ( \eta^2 u^{1- \gamma}) \rangle dx \nonumber \\
         =& -\int_{B_{2r}(x_0)} \eta^2 (1-\gamma) u^{-\gamma} \langle A(x)\nabla u, \nabla u\rangle |^2 + 2\eta u^{1-\gamma} \langle A\nabla u , \nabla \eta \rangle\, dx \nonumber
    \end{align*}
    Rewriting in terms of $w:= u^{(2-\gamma)/2}$ yields, 
    \begin{align*}
         &\gamma \int_{B_{2r}(x_0)}  f(x) \eta^2 dx\\
         =&  - \int_{B_{2r}(x_0)} \eta^2 \frac{4 (1-\gamma)}{(2 - \gamma)^2}  \langle A\nabla w , \nabla w\rangle  +  \frac{4}{(2 - \gamma)}\eta w \langle A \nabla w , \nabla \eta \rangle\, dx \nonumber \\
         \le& - \frac{4 (1-\gamma)\lambda}{(2 - \gamma)^2} \int_{B_{2r}(x_0)} \eta^2   |\nabla w |^2+  
         \frac{2 \Lambda}{(2 - \gamma)}  \int_{B_{2r}(x_0)} (\delta^2 w^2 |\nabla \eta |^2 + \frac{1}{\delta^2} \eta^2 |\nabla w|^2) \nonumber \\
          \le & \frac{2 \Lambda}{(2 - \gamma)}  \int_{B_{2r}(x_0)} \delta^2 w^2 |\nabla \eta|^2 dx \nonumber
    \end{align*}
by choosing $\delta > 0$ so that 
    $$
        \frac{1}{\delta^2} \frac{2\Lambda}{(2 - \gamma)}  = \frac{4(1-\gamma)\lambda}{(2 - \gamma)^2}. 
    $$
    and using Cauchy-Schwartz inequality. Therefore,
    \begin{align}\label{key-est-n}
         \gamma \int_{B_{2r}(x_0)}  f(x) \eta^2 dx & \le \frac{2\Lambda \delta^2}{(2 - \gamma)} \int_{B_{2r}(x_0)}  w^2 |\nabla \eta|^2 dx  \\
         &\le c(\Lambda, \gamma) r^{-2} \int_{B_{\frac 32r}(x_0)} w^2 dx  \nonumber \\
         &=  c(\Lambda, \gamma) r^{-2} \int_{B_{\frac 32r}(x_0)} u^{2- \gamma} dx \nonumber \\
          \end{align}

Since $f\ge C_0$, we have 
\[\int_{B_{\frac 32r}(x_0)}u^{2-\gamma}\ge C(\Lambda,\gamma, C_0,n)r^{n+2}.\]
Therefore there exist $x_1\in B_{\frac 32 r}(x_0)$ such that 
\[u(x_1)^{2-\gamma}\ge \fint_{B_{\frac 32 r}(x_0)}u^{2-\gamma}>Cr^2.\]
Notice that the distance from $x_1$ to the free boundary is comparable to $r$, thus
Theorem \ref{non-deg theom} is established. $\Box$

\medskip

\noindent{\bf Proof of Theorem \ref{non-de-2}:}   The difference of the proof mainly starts from (\ref{key-est-n}). Of course in this theorem we know $y_0\in \partial \{u>0\}$ and $B_{4r}(y_0)\in B_1$. We start from (\ref{key-est-n}).     Now we estimate the left hand side differently:
         \[\gamma \int_{B_{2r}(x_0)}f(x)\eta^2dx=\gamma \int_{B_{2r}(x_0)}f_+\eta^2dx-\gamma \int_{B_{2r}(x_0)}f_-\eta^2dx.\]
    \begin{align*}
         &\gamma \int_{B_{2r}(x_0)}  f_+(x) \eta^2 dx 
         \ge \gamma \int_{B_{r}(x_0)}  f_+(x)  dx \nonumber \\
         \ge& \gamma r^{-n/p} \left| \{ f  > r^{-n/p} \} \cap B_{r} (x_0)  \right|\nonumber \\
         =& \gamma r^{-n/p}  \left( \left| \{ f  > r^{-n/p} \} \cap B_{4r} (y_0) \right| - \left| \{ f  > r^{-n/p} \} \cap B_{4r} (y_0) \setminus B_r(x_0) \right| \right) \nonumber \\
         \ge& \gamma r^{-n/p}  \left( \tau  - \iota \right) |B_{4r} (y_0) |\nonumber
    \end{align*} 
   where
    $$
        \iota : = \frac{| B_{4r} (y_0) \setminus B_r(x_0) | }{|B_{4r}(y_0)|} = \frac{(4^n - 1)}{4^n} . 
    $$
    Now choose $ \iota < \tau   < 1$. On the other hand, 
    \[-\gamma \int_{B_{2r}(x_0)}f_-\eta^2dx\ge -\varepsilon r^{n-\frac{n}p}.\]
    
    Therefore, we reach 
    \begin{equation}
        \tilde{c}(\Lambda, \gamma, n) r^{(2 - n/p) } \le \fint_{B_{\frac 32 r}(x_0)} u^{2 - \gamma} dx 
    \end{equation}
    Then just like the proof of Theorem \ref{non-deg theom} we find $x_1\in B_{\frac 32 r}(x_0)$ whose distance to the free boundary is obviously comparable to $r$ such that
    \begin{equation}
          \tilde{c}(\Lambda, \gamma, n) r^{(2 - n/p) }  \le  u(x_1)^{2 - \gamma}.  \nonumber
    \end{equation}
 Since $x_0$ was taken arbitrary on $\partial B_{2r}(y_0)$, then by taking the supremum we reach \eqref{non-degen-add-2}. 
Theorem \ref{non-de-2} is established. 
\end{proof}

Theorem \ref{non-de-2} provides a non-degeneracy rate strong enough as to guarantee positive density at free boundary points where the source function has sufficient mass concentrated.

\begin{theorem}(Positive density)
    Assume $u$ is a minimizer of \eqref{main functional} and let $y_0 \in \partial \{ u > 0\} \cap B_{1/2}$ such that \eqref{PosMass} holds, for some $\tau , r_0 \in (0,1)$. Then there exists $\tau_0 > 0$, depending on $n, \tau, \gamma$,  such that 
    \begin{equation}
        \frac{|B_r(y_0) \cap \{ u > 0 \} |}{|B_r(y_0)|} \ge \tau_0
    \end{equation}
    for all $ 0 < r < r_0$. 
\end{theorem}

\begin{proof}
    The proof follows a standard geometric argument, and we carry it here for completeness. Let $ 0 < r \le \min\{ \frac{\rho_*}{4}, r_0\}$, where $\rho_*, r_0 > 0$ are the constants granted by Theorem \ref{improv reg} and Theorem \ref{non-de-2}, respectively. By the non-degeneracy (Theorem \ref{non-de-2}) there exists $y \in \partial B_r (y_0)$ such that 
    \begin{equation}\label{non-deg - 3}
        u(y) \ge \tilde{c} r^{\theta}, \quad \theta = \frac{2 - n/p}{2 - \gamma}. 
    \end{equation}
    Let $z \in \partial \{ u > 0 \}$ such that $d: = |y-z| = dist(y, \partial\{ u > 0\} ). $ Thus from the sharp regularity (Theorem \ref{improv reg}), we have
    \begin{equation}\label{thm2 + non-deg}
         \tilde{c} r^{\theta} \le u(y) \le \sup_{B_{2d}(z)} u \le C d^{\theta} .
    \end{equation}
     This implies
    $$
        r \le \left( \frac{C}{\tilde{c}}\right)^{\frac{1}{\theta}} d \le \max\{ 1, \left( \frac{C}{\tilde{c}}\right)^{\frac{1}{\theta}} \} \, d.
    $$
    Next, by letting $\mu = \min\{ 1, \left( \frac{C}{\tilde{c}}\right)^{\frac{1}{\theta}}  \}$, then 
    $$
        B_{\mu \, r} (y) \subset B_d (y) \subset \{ u > 0 \}.
    $$
    Moreover, since $d \le r$ then for $z \in B_d(y)$, $|z- y_0| \le |z - y| + | y - y_0| < d + r \le 2r$, thus
    $$
     B_{\mu \, r} (y) \subset B_d (y) \subset B_{2r}(z_0).
    $$
    We conclude that 
    \begin{eqnarray}
        |B_{2r}(y_0) \cap \{ u > 0\} | &\ge& |B_{\mu \, r}(y)| \nonumber \\
        &=& \mu^n r^n \alpha(n) \nonumber \\
        &=& \left( \frac{\mu}{2}\right)^n (2r)^n \alpha(n), \nonumber
    \end{eqnarray}
    where $\alpha(n) = |B_1|$. Thus, by taking $\tau_0 = \left( \frac{\mu}{2}\right)^n$ the proof is complete.

\end{proof}

\section{A radial example} \label{sct radial example}

In this final section, we construct a radial solution to illustrate 
the sharpness of our regularity and non-degeneracy results. By 
considering a source term with a calibrated singularity at the 
origin, we demonstrate that the $L^{p,\infty}$ scaling is the 
critical threshold for maintaining the geometric properties of the 
free boundary.

 Let 
\[
u(x) = |x|^{\alpha}, \qquad \alpha > 0.
\] 
A direct computation shows that $u$ satisfies 
\[
\Delta u = f(x) u^{\gamma -1}, \quad \text{in} \quad \{u > 0\}
\]
with
\[
f(x) = \alpha (n + \alpha - 2)\, |x|^{\alpha(2 - \gamma) - 2}.
\]

To analyze the integrability properties of $f$, we introduce the exponent
\[
-\beta := \alpha(2 - \gamma) - 2.
\]
In this notation,
\[
f(x) \simeq |x|^{-\beta}.
\]
For each $t > 0$, the distribution function of $f$ can be computed explicitly:
\[
\{ f(x) > t \} = \{ |x|^{-\beta} > t \} = \big\{ |x| < t^{-1/\beta} \big\}.
\]
Hence,
\[
t \, \big| \{ f(x) > t \} \big|^{1/p}
= t \, |B_{t^{-1/\beta}}|^{1/p}
= |B_1|^{1/p}\, t^{\,1 - \tfrac{n}{\beta p}}.
\]
From the characterization of weak $L^p$ spaces, it follows that 
\[
f \in L^{p,\infty} (B_1) \quad \Longleftrightarrow \quad 1 - \tfrac{n}{\beta p} = 0,
\]
that is,
\[
\beta = \frac{n}{p}.
\]
Recalling the definition of $\beta$, this condition is equivalent to
\[
\alpha = \frac{2 - n/p}{2 - \gamma}.
\]



\bibliographystyle{plain}
\bibliography{biblio2}

\end{document}